\documentclass[11pt]{amsart}
\usepackage{mathrsfs}
\usepackage{amsmath}
\usepackage{amsfonts}
\usepackage{amssymb}
\usepackage{amsxtra}
\usepackage{mathtools}
\usepackage{dsfont}

\usepackage{graphicx}
\usepackage[T1]{fontenc}
\usepackage{newtxmath}
\usepackage{newtxtext}
\usepackage[margin=2cm, centering]{geometry}
\usepackage{enumitem}

\usepackage[english,polish]{babel}

\usepackage[colorlinks,citecolor=blue,urlcolor=blue,bookmarks=true]{hyperref}
\hypersetup{
pdfpagemode=UseNone,
pdfstartview=FitH,
pdfdisplaydoctitle=true,
pdfborder={0 0 0}, 
pdftitle={Endpoint estimates for the maximal function over prime numbers}
pdfauthor={Bartosz Trojan},
pdflang=en-US
}

\newcommand{\TT}{\mathbb{T}}
\newcommand{\CC}{\mathbb{C}}
\newcommand{\ZZ}{\mathbb{Z}}
\newcommand{\NN}{\mathbb{N}}
\newcommand{\RR}{\mathbb{R}}
\newcommand{\PP}{\mathbb{P}}

\newcommand{\QQ}{\mathbb{Q}}

\newcommand{\calF}{\mathcal{F}}

\newcommand{\calO}{\mathcal{O}}
\newcommand{\calM}{\mathcal{M}}
\newcommand{\calB}{\mathcal{B}}

\newcommand{\calA}{\mathcal{A}}

\newcommand{\scrR}{\mathscr{R}}

\newcommand{\scrA}{\mathscr{A}}

\newcommand{\pl}[1]{\foreignlanguage{polish}{#1}}

\newcommand{\abs}[1]{\lvert {#1} \rvert}

\newcommand{\ind}[1]{{\mathds{1}_{{#1}}}}

\renewcommand{\atop}[2]{\genfrac{}{}{0pt}2{#1}{#2}}

\newtheorem{theorem}{Theorem}[section]
\newtheorem{proposition}[theorem]{Proposition}
\newtheorem{lemma}[theorem]{Lemma}
\newtheorem{corollary}[theorem]{Corollary}
\newtheorem{claim}[theorem]{Claim}
\newtheorem*{theorem*}{Theorem}

\theoremstyle{plain}
\newcounter{thm}

\newtheorem{main_theorem}[thm]{Theorem}

\theoremstyle{definition}

\title[Endpoint estimates]
{Endpoint estimates for the maximal function over prime numbers}

\author{Bartosz Trojan}
\address{
	\pl{
	Bartosz Trojan\\
	Institute of Mathematics of Polish Academy of Science\\
	ul. \'Sniadeckich 8\\
	00-656 Warszawa\\
	Poland}
}
\email{btrojan@impan.pl}

\begin{document}
\selectlanguage{english}

\begin{abstract}
	Given an ergodic dynamical system $(X, \calB, \mu, T)$, we prove that for each function $f$ belonging to the
	Orlicz space $L(\log L)^2(\log \log L)(X, \mu)$, the ergodic averages
	\[
		\frac{1}{\pi(N)} \sum_{p \in \PP_N} f\big(T^p x\big),
	\]
	converge for $\mu$-almost all $x \in X$, where $\PP_N$ is the set of prime numbers not larger that $N$ and
	$\pi(N) = \# \PP_N$.
\end{abstract}


\maketitle

\section{Introduction}
Let $(X, \calB, \mu, T)$ be an ergodic dynamical system, that is $(X, \calB, \mu)$ is a probability space with
a measurable and measure preserving transformation $T: X \rightarrow X$. The classical Birkhoff theorem \cite{birk} 
states that for any function $f$ from $L^p(X, \mu)$ with $p \in [1, \infty)$, the ergodic averages 
\[
	\frac{1}{N}\sum_{n = 0}^{N-1} f\big(T^n x\big)
\]
converge for $\mu$-almost all $x \in X$. This classical result, among others, motivates studying ergodic averages
over subsequences of integers. In this article we are interested in pointwise convergence of the following averages,
\[
    \scrA_N f(x) = \frac{1}{\pi(N)} \sum_{p \in \PP_N} f\big(T^p x \big)
\]
where $\PP_N$ is the set of prime numbers not larger than $N$ and $\pi(N) = \# \PP_N$. The problem of ergodic averages
along prime numbers was initially studied by Bourgain in \cite{bou-p} where the case of functions belonging to
$L^2(X, \mu)$ has been covered. It was extended by Wierdl in \cite{wrl} to all $L^p(X, \mu)$, for $p > 1$, see also
\cite[Section 9]{bou}. However, the endpoint $p = 1$, was left open for more than twenty years. Following the method
developed in \cite{BM} by Buczolich and Mauldin, LaVictoire in \cite{LaV} has shown that for each ergodic dynamical
system there exists $f \in L^1(X, \mu)$ such that the sequence $(\scrA_N f : N \in \NN)$ diverges on a set of positive
measure. 

The purpose of this article is to find an Orlicz space close to $L^1(X, \mu)$ where the almost
everywhere convergence holds. We show the following theorem (see Theorem \ref{thm:4}).
\begin{main_theorem}
	\label{main_thm:1}
	For each $f \in L(\log L)^2(\log \log L)(X, \mu)$, the limit
	\[
		\lim_{N \to \infty} \scrA_N f(x)
	\]
	exists for $\mu$-almost all $x \in X$.
\end{main_theorem}
In light of the pointwise convergence obtained by Bourgain in \cite{bou1}, see also \cite{mtz}, to prove Theorem
\ref{main_thm:1} it suffices to show the weak maximal ergodic inequality for functions in Orlicz space
$L(\log L)^2(\log \log L)(X, \mu)$. This inequality is deduce from the following restricted weak Orlicz estimate.
\begin{main_theorem}
	\label{main_theorem:3}
	There is $C > 0$ such that for any subset $A \subset X$, 
	\[
		\mu\Big\{x \in X : \sup_{N \in \NN} \scrA_N\big(\ind{A}\big)(x) > \lambda \Big\}
		\leq
		C \lambda^{-1} \log^2(e/\lambda) \mu(A)
	\]
	for all $1 > \lambda > 0$.
\end{main_theorem}
By appealing to the Calder\'on transference principle, see \cite{cald}, Theorem \ref{main_theorem:3} is deduced from
the corresponding result for integers $\ZZ$ with the counting measure and the shift operator. To be more precise, for a
function $f: \ZZ \rightarrow \CC$, we define
\[
	\calA_N f(x) = \frac{1}{\pi(N)} \sum_{p \in \PP_N} f(x+p).
\]
Our main result is following theorem (see Theorem \ref{thm:5}). 
\begin{main_theorem}
	\label{main_thm:2}
	There is $C > 0$ such that for any subset $F \subset \ZZ$ of a finite cardinality
	\[
		\Big|\Big\{
		x \in \ZZ : \sup_{N \in \NN} \calA_N \big(\ind{F} \big)(x) > \lambda
		\Big\}\Big|
		\leq
		C
		\lambda^{-1} \log^2(e/\lambda) \abs{F}
	\]
	for all $0 < \lambda < 1$.
\end{main_theorem}
Theorem \ref{main_thm:2} together with $\ell^2(\ZZ)$ estimates are sufficiently strong to imply the maximal inequality for
all $\ell^p(\ZZ)$ spaces, for $p > 1$, giving an alternative proof of the Wierld's theorem \cite{wrl}.

Let us now give some details about the proof of Theorem \ref{main_thm:2}. Without loss of generality, we may restrict the
supremum to dyadic numbers. It is more convenient to work with weighted averages $\calM_N f$ instead of $\calA_N f$
where
\[
	\calM_N f(x) = \frac{1}{\vartheta(N)} \sum_{p \in \PP_N} f(x+p) \log p,
\]
and
\[
	\vartheta(N) = \sum_{p \in \PP_N} \log p.
\]
Given $t > 0$, for each $n \in \NN$, we decompose the operator $\calM_{2^n}$ into two parts $A_n^t$ and $B_n^t$,
in such a way that the maximal function associated with $A_n^t$ has $\ell^{1,\infty}(\ZZ)$ norm
$\lesssim t \|f\|_{\ell^1}$, whereas the one corresponding to $B_n^t$ has $\ell^2(\ZZ)$ norm
$\lesssim \exp\big(-c \sqrt{t}\big) \|f\|_{\ell^2}$. When applied to the distribution function
$\big|\big\{\sup_{n \in \NN} \calM_{2^n}(\ind{F}) > \lambda \big\}\big|$, we can optimize both estimates by taking
$t \simeq \log^2(e/\lambda)$. This idea originated to Ch. Fefferman \cite{feff}, see also Bourgain \cite{bou3}. 
Ionescu introduced this technique in a related discrete context, see \cite{I}.
The decomposition of $\calM_{2^n}$ uses the circle method of Hardy and Littlewood. However,
to achieve the exponential decay of the error term, due to the Page's theorem, the approximating multiplier has to 
contain the second term of the asymptotic as well. Thus, the possible existence of the Siegel zero entails that in
the neighborhood of the rational point $a/q$ the approximating multiplier $\widehat{L^{a, q}_{2^n}}(\cdot - a/q)$ 
depends on the rational number $a/q$. We refer to Sections \ref{sec:4} and \ref{sec:2} for details. Thanks to the
log-convexity of $\ell^{1, \infty}(\ZZ)$, the weak type estimates are reduced to showing
\[
	\Big|\Big\{
	x \in \ZZ:
	\sup_{t \leq n}
	\Big|
	\sum_{a \in A_q}
	\calF^{-1}\big(\widehat{L^{a, q}_{2^n}}(\cdot - a/q) \eta_s(\cdot - a/q) \hat{f} \big)(x)
	\Big| > \lambda
	\Big\}\Big|
	\leq
	C \frac{1}{\lambda \varphi(q)} \|f\|_{\ell^1}
\]
for $2^s \leq q < 2^{s+1}$ with $1 \leq s \leq \sqrt{t}$. At this stage we exploit the behavior of the Gauss sums
described in Theorem \ref{thm:7}.

Let us emphasize that under the Generalized Riemann Hypothesis we can obtain in Proposition
\ref{prop:1}, and consequently in Theorem \ref{thm:1}, a better error estimate. However, it is not clear whether one can
prove Theorem \ref{thm:3} with the bounds proportional to $\sqrt{t} \|f\|_{\ell^1}$.

The paper is organized as follows. In Section \ref{sec:3}, we collect necessary facts about Dirichlet characters
and the zero-free region. Then we evaluate the Gauss sum that appears in the approximating multiplier
(Theorem \ref{thm:7}). Section \ref{sec:4} is devoted to construction of the approximating multipliers. In Sections
\ref{sec:2} and \ref{sec:1}, we show $\ell^2$ and the weak type estimates, respectively. In Section \ref{sec:6},
we give two applications of Theorem \ref{main_thm:2}. Namely, we show how to deduce the maximal ergodic inequality for
functions from $\ell^p(\ZZ)$, (Theorem \ref{thm:8}). Next we apply the transference principle 
(Proposition \ref{prop:2}) and show almost everywhere convergence of the ergodic averages $(\scrA_N f : N \in \NN)$
for $f \in L(\log L)^2(\log \log L)(X, \mu)$, (Theorem \ref{thm:4}).

\subsection*{Notation}
Throughout the whole article, we write $A \lesssim B$ ($A \gtrsim B$) if there is an absolute constant $C>0$ such that
$A\le CB$, ($A\ge CB$). Moreover, $C$ stands for a large positive constant which value may vary from occurrence to
occurrence. If $A \lesssim B$ and $A\gtrsim B$ hold simultaneously then we write $A \simeq B$. The set of positive integers
and the set of prime numbers are denoted by $\NN$ and $\PP$, respectively. For $x > 0$, we set $\ZZ_x = [1, x] \cap \NN$.
Let $\NN_0 = \NN \cup \{0\}$. 

\section{Gauss sums}
\label{sec:3}
We start by recalling some basic facts from number theory. A general reference here is the book \cite{mv}.

A homomorphism
\[
	\chi: \big(\ZZ / q\ZZ\big)^\times \rightarrow \CC^\times,
\]
is called a Dirichlet character modulo $q$. The simplest example, called the \emph{principal} character modulo $q$,
is defined as
\[
	\mathds{1}_q(x) = 
	\begin{cases}
		1 & \text{if } \gcd(x, q) = 1, \\
		0 & \text{otherwise.}
	\end{cases}
\]
A character $\chi$ modulo $q$ is \emph{primitive}, if $q$ is the least integer $d$, such that $\chi(m) = \chi(n)$ for all
$m \equiv n \pmod d$ and $(mn, q) = 1$. For each character $\chi$ there is the unique primitive character $\chi^\star$
modulo $q_0$ for some $q_0 \mid q$, such that
\[
	\chi(n) = 
	\begin{cases}
		\chi^\star(n) & \text{if } (n, q) = 1, \\
		0 & \text{otherwise}.
	\end{cases}
\]
The character is \emph{quadratic} if it takes only values $\{-1, 0, 1\}$ with at least one $-1$. Recall that, if
$\chi^\star$ is a primitive quadratic character with modulus $q_0$, then 
\begin{itemize}
	\item $q_0 \equiv 1 \pmod 4$, and $q_0$ is square-free, or
	\item $4 \mid q_0$, $q_0/4 \equiv 2 \text{ or } 3 \pmod 4$, and $q_0/4$ is square-free. 
\end{itemize}
Given a Dirichlet character $\chi$ and $s \in \CC$ with $\Re s > 1$, we define the Dirichlet $L$-function by the formula
\[
	L(s, \chi) = \sum_{n \geq 1} \frac{\chi(n)}{n^s}.
\]
In fact, $L(\: \cdot\:, \chi)$ extends to the analytic function in $\{z \in \CC : \Re z > 0\}$. There is an absolute
constant $c > 0$, such that if $\chi$ is a Dirichlet character modulo $q$, then the region
\begin{equation}
	\label{eq:9}
	\Big\{
	z \in \CC : 1 - \frac{c}{\log q} < \Re z < 1
	\Big\}
\end{equation}
contains at most one zero of $L(\: \cdot \:, \chi)$, which we denote by $\beta_q$. The zero $\beta_q$ is real and the
corresponding character is quadratic. The character having zero in \eqref{eq:9} is called \emph{exceptional}. Since
$L(\beta, \chi) = 0$ implies that $L(1-\beta, \chi) = 0$, we may assume that $\frac{1}{2} \leq \beta_q < 1$.

The Gauss sum of a Dirichlet character $\chi$ modulo $q$ is defined as
\[
	G(\chi, n) = \frac{1}{\varphi(q)} \sum_{r \in A_q} \chi(r) e^{2\pi i r n / q}
\]
where $A_q = \big\{1 \leq a \leq q : \gcd(a, q) = 1\big\}$, and $\varphi(q) = \# A_q$. Let us recall that for each
$\epsilon > 0$ there is $C_\epsilon > 0$ such that
\begin{equation}
	\label{eq:35}
	\varphi(q) \geq C_\epsilon q^{1 - \epsilon}.
\end{equation}
We set
\[
    \tau(\chi) = \varphi(q) G(\chi, 1).
\]
Let us denote by $\mu$ the M\"obious function, which is defined for $q = p_1^{\alpha_1} \dots p_n^{\alpha_n}$,
where $p_1, \ldots, p_n$ are distinct primes, as
\[
	\mu(q) = 
	\begin{cases}
		(-1)^n & \text{if }\alpha_1 = \ldots = \alpha_n = 1, \\
		0 & \text{otherwise,}
	\end{cases}
\]
and $\mu(1) = 1$. The following theorem plays the crucial role in Section \ref{sec:1}.
\begin{theorem}
	\label{thm:7}
	Let $\chi$ be a quadratic Dirichlet character modulo $q$ induced by $\chi^\star$ having the conductor $q_0$.
	For $x \in \ZZ$, we set $r = \gcd(q, x)$. Then
	\[
		\sum_{a \in A_q} G(\chi, a) e^{2\pi i x a/q} = \mu(r) q_0 \frac{\varphi(r)}{\varphi(q)} \chi^\star(-x)
	\]
	provided that $q/q_0$ is square-free, $\gcd(q/q_0, q_0) = 1$ and $r \mid q/q_0$. Otherwise the sum equals zero.
\end{theorem}
\begin{proof}
	By \cite[Theorem 9.12]{mv}, if $r \mid q/q_0$ then
	\begin{equation}
		\label{eq:37}
		\sum_{a \in A_q} \chi(a) e^{2\pi i a x/q} =
		\frac{\varphi(q)}{\varphi(q/r)} \chi^\star\big(x/r\big)
		\chi^\star\big(q/(r q_0)\big) \mu\big(q/(r q_0)\big) \tau(\chi^\star),
	\end{equation}
	otherwise the sum equals zero. In particular, for $a \in A_q$, we have
	\begin{equation}
		\label{eq:34}
    	G(\chi, a) = \frac{\mu(q/q_0)}{\varphi(q)} \chi^\star(a) \chi^\star(q/q_0) \tau(\chi^\star).
	\end{equation}
	Hence, $G(\chi, a) \neq 0$ entails that $q/q_0$ is square-free and $\gcd(q/q_0, q_0) = 1$. Next, using
	\eqref{eq:34} and \eqref{eq:37} we get
	\begin{align*}
		\sum_{a \in A_q} G(\chi, a) e^{2\pi i x a/q}
		&= \frac{\mu(q/q_0)}{\varphi(q)} \chi^\star(q/q_0) \tau(\chi^\star)
		\sum_{a \in A_q} \chi(a) e^{2\pi i x a /q} \\
		&= \frac{\mu(r)}{\varphi(q/r)} \chi^\star(q/q_0) 
		\chi^\star\big(x/r\big)
        \chi^\star\big(q/(r q_0)\big) 
		\tau(\chi^\star)^2 \\
		&=
		\frac{\mu(r)}{\varphi(q/r)}
		\chi^\star(x) \tau(\chi^\star)^2.
	\end{align*}
	Because $\abs{\tau(\chi^\star)} = \sqrt{q_0}$, we have $\tau(\chi^\star)^2 = q_0 \chi^\star(-1)$. Hence,
	\begin{equation}
		\label{eq:40}
		\sum_{a \in A_q} G(\chi, a) e^{2\pi i x a/q} = \frac{\mu(r)}{\varphi(q/r)} \chi^\star(-x) q_0.
	\end{equation}
	Finally, since $q/q_0$ is square-free, $\gcd(q/q_0, q_0) = 1$ and $r \mid q/q_0$, we deduce that $\gcd(q/r, r) = 1$.
	Therefore,
	\[
		\varphi(q/r) \varphi(r) = \varphi(q),
	\]
	which together with \eqref{eq:40} completes the proof.
\end{proof}
Let us observe that the identity \eqref{eq:34} together with \eqref{eq:35} imply that
\begin{equation}
	\label{eq:15}
	\big|G(\chi, a)\big| \leq \frac{\sqrt{q_0}}{\varphi(q)} \leq C_\epsilon q^{-\frac{1}{2}+\epsilon}.
\end{equation}
for any $\epsilon > 0$. Moreover, $G(\chi, a) \neq 0$ entails that $q$ is square-free or $4 \mid q$ and $q/4$
is square-free.

\section{Approximating multipliers}
\label{sec:4}
Let us denote by $\calA_N$ the averaging operator over prime numbers, that is for a function $f: \ZZ \rightarrow \CC$
we have
\[
	\calA_N f(x) = \frac{1}{\pi(N)} \sum_{p \in \PP_N} f(x+p)
\]
where $\PP_N = [1, N] \cap \PP$ and $\pi(N) = \# \PP_N$. Since sums over primes are very irregular, it is more
convenient to work with
\[
	\calM_N f(x)=\frac{1}{\vartheta(N)} \sum_{p \in \PP_N} f(x+p) \log p
\]
where
\[
	\vartheta(N) = \sum_{p \in \PP_N} \log p.
\]
By the partial summation, we easily see that
\begin{align*}
	\sum_{p \in \PP_N} f(x+p)
	&= \sum_{n=2}^N \Big(\vartheta(n) \calM_n f(x) - \vartheta(n-1) \calM_{n-1}f(x)\Big) \frac{1}{\log n} \\
	&= \vartheta(N) \calM_N f(x) \frac{1}{\log N} 
	+ \sum_{n = 2}^{N-1} \vartheta(n) \calM_n f(x) \bigg(\frac{1}{\log n} - \frac{1}{\log(n+1)}\bigg),
\end{align*}
thus
\begin{align}
	\nonumber
	\big| \calA_N f(x) \big|
	&\leq
	\sup_{N' \in \NN} \big| \calM_{N'} f(x) \big|
	\frac{1}{\pi(N)} 
	\bigg(\vartheta(N) \frac{1}{\log N} + \sum_{n = 2}^{N-1} \vartheta(n) 
	\bigg(\frac{1}{\log n} - \frac{1}{\log(n+1)}\bigg) \bigg) \\
	\label{eq:39}
	&\leq
	\sup_{N' \in \NN} \big| \calM_{N'} f(x) \big|.
\end{align}
To better understand the operators $\calM_N$, we use the Hardy--Littlewood circle method. Let $\calF$ denote the Fourier
transform on $\RR$ defined for any function $f \in L^1(\RR)$ as
\[
	\calF f(\xi) = \int_\RR f(x) e^{-2\pi i \xi x} {\: \rm d} x.
\]
If $f \in \ell^1(\ZZ)$, we set
\[
	\hat{f}(\xi) = \sum_{n \in \ZZ} f(n) e^{-2\pi i \xi n}.
\]
To simplify the notation we denote by $\mathcal F^{-1}$ the inverse Fourier transform on $\RR$ or the inverse Fourier
transform on the torus $\TT\equiv[0, 1)$, depending on the context. Let $\mathfrak{m}_N$ be the Fourier multiplier
corresponding to $\calM_N$, i.e.,
\begin{equation}
	\label{eq:29}
	\mathfrak{m}_N(\xi) = \frac{1}{\vartheta(N)} \sum_{p \in \PP_N} e^{2\pi i \xi p} \log p.
\end{equation}
Then for a finitely supported function $f: \ZZ \rightarrow \CC$, we have
\[
	\calM_N f(x) = \calF^{-1}\big(\mathfrak{m}_N \hat{f} \big)(x).
\]
For $\frac{1}{2} \leq \beta \leq 1$, we set
\begin{equation}
	\label{eq:43}
	M_N^\beta = \frac{1}{N} \sum_{n = 1}^N \frac{n^\beta - (n-1)^\beta}{\beta} \delta_n. 
\end{equation}
To simplify the notation we write $M_N$ for $M_N^1$. Let $M_0 \equiv 0$. Recall that
\begin{equation}
	\label{eq:45}
	\big|\widehat{M_N}(\xi) \big| \lesssim \min\Big\{\big(N \abs{\xi}\big)^{-1}, N\abs{\xi}\Big\}.
\end{equation}
For $\beta < 1$, we notice that the operators $M_N^\beta$ are not averaging operators. Moreover, by the partial summation
and \eqref{eq:45}, we get
\begin{align*}
	\big|\widehat{M_N^\beta} (\xi) \big| 
	&= \frac{1}{\beta N} \Big|\sum_{n=1}^N \big(n \widehat{M_n}(\xi) - (n-1) \widehat{M_{n-1}}(\xi)\big) 
	\big(n^\beta - (n-1)^\beta\big)\Big| \\
	&\lesssim
	\big(N \abs{\xi} \big)^{-1} 
	N^{\beta-1} + \big(N\abs{\xi}\big)^{-1} \sum_{n = 1}^{N-1} 
	\big(2 n^\beta - (n-1)^\beta - (n+1)^\beta\big) \\
	&\lesssim
	\big(N \abs{\xi} \big)^{-1} N^{\beta-1}
	+
	\big(N\abs{\xi}\big)^{-1} \sum_{n = 1}^{N-1} n^{\beta-2}
	.
\end{align*}
Hence,
\begin{equation}
	\label{eq:14}
	\big|\widehat{M_N^\beta}(\xi) \big| \lesssim \big(N \abs{\xi}\big)^{-1}.
\end{equation}
Moreover, 
\[
	\big|\widehat{M_N^\beta}(\xi) - \beta^{-1} N^{\beta-1} \big| \lesssim N^\beta \abs{\xi}, 
\]
thus
\begin{align*}
	\big|\widehat{M_N^{\beta}}(\xi) - \widehat{M_{2N}^{\beta}}(\xi)\big| 
	&\lesssim
	\big|\widehat{M_N^{\beta}}(\xi) - \beta^{-1} N^{\beta-1} \big|
	+
	\big|\widehat{M_{2N}^\beta}(\xi) - \beta^{-1} (2N)^{\beta-1}\big|
	+
	\big|\beta^{-1} N^{\beta-1} + \beta^{-1} (2N)^{\beta-1}\big| \\
	&\lesssim 
	N^\beta \abs{\xi} + (1-\beta) N^{\beta-1}.
\end{align*}
Therefore,
\begin{align}
	\nonumber
	\big|\widehat{M_N^{\beta}}(\xi) - \widehat{M_{2N}^{\beta}}(\xi)\big|
	&\lesssim
	\min\Big\{(N \abs{\xi})^{-1}, N^\beta \abs{\xi} + (1-\beta) N^{\beta-1} \Big\} \\
	\label{eq:24}
	&\lesssim
	\min\Big\{(N \abs{\xi})^{-1}, N \abs{\xi} \Big\} + (1-\beta) N^{\beta-1}.
\end{align}
Given $q \in \NN$, and $a \in A_q$, we set
\begin{equation}
	\label{eq:41}
	L_N^{a, q} = G(\mathds{1}_q, a) M_N,
\end{equation}
if there is no exceptional character modulo $q$, and
\begin{equation}
	\label{eq:42}
	L_N^{a, q} = G(\mathds{1}_q, a) M_N - G(\chi_q, a) M_N^{\beta_q},
\end{equation}
when there is an exceptional character $\chi_q$ modulo $q$ and $\beta_q$ is the corresponding zero.

\begin{proposition}
	\label{prop:1}
	There is $c > 0$ such that if $\xi \in \TT$,
	\[
		\bigg| \xi - \frac{a}{q} \bigg| \leq N^{-1} Q
	\]
	for some $1 \leq q \leq Q$, $a \in A_q$, and $1 \leq Q \leq \exp\big(c\sqrt{\log N}\big)$, then
	\[
		\mathfrak{m}_N (\xi) 
		= \widehat{L^{a,q}_N}(\xi - a/q) 
		+\calO\Big(Q \exp\big(-c\sqrt{\log N}\big)\Big).
	\]
\end{proposition}
\begin{proof}
	Observe that for a prime $p$, $p \mid q$ if and only if $(p \bmod q, q) > 1$. Hence,
	\begin{align*}
		\Big|
		\sum_{\atop{r = 1}{(r, q) > 1}}^q 
		\sum_{\atop{p \in \PP_N}{p \equiv r \bmod q}}
		e^{2\pi i \xi p}
		\log p
		\Big|
		\leq
		\sum_{\atop{p \in \PP}{p \mid q}} \log p
		\leq
		q.
	\end{align*}
	Let $\theta = \xi - a/q$. For $p \equiv r \pmod q$, we have
	\[
		\xi p \equiv \theta p + r a / q \pmod 1,
	\]
	thus
	\[
		\sum_{r \in A_q} \sum_{\atop{p \in \PP_N}{p \equiv r \bmod q}}
		e^{2 \pi i \xi p} \log p
		=
		\sum_{r \in A_q} 
		e^{2\pi i r a/q} \sum_{\atop{p \in \PP_N}{p \equiv r \bmod q}}
        e^{2\pi i \theta p} \log p.
	\]
	For $x \geq 2$, we set
	\[
		\vartheta(x; q, r) = \sum_{\atop{p \in \PP_x}{p \equiv r \bmod q}} \log p.
	\]
	Then, by the partial summation, we obtain
	\begin{align}
		\nonumber
		\sum_{\atop{p \in \PP_N}{p \equiv r \bmod q}}
        e^{2 \pi i \theta p} \log p
		&=
		\sum_{\atop{p \in \PP_N \setminus \PP_{\sqrt{N}}}{p \equiv r \bmod q}}
		e^{2\pi i \theta p} \log p
		+
		\calO\big(\sqrt{N}\big) \\
		\label{eq:10}
		&=
		\vartheta(N; q, r) e^{2\pi i \theta N} 
		- \vartheta(\sqrt{N}; q, r) e^{2 \pi i \theta \sqrt{N}}
		- 2 \pi i \theta \int_{\sqrt{N}}^N 
		\vartheta(t; q, r) 
		e^{2\pi i \theta t} {\: \rm d} t + \calO\big(\sqrt{N} \big).
	\end{align}
	Analogously, for any $\frac{1}{2} \leq \beta \leq 1$, we can write
	\begin{equation}
		\label{eq:11}
		\sum_{n = 1}^N \frac{n^\beta - (n-1)^\beta}{\beta} e^{2\pi i \theta n}
		= 
		\beta^{-1} N^\beta e^{2\pi i \theta N} 
		- \beta^{-1} \sqrt{N^\beta} e^{2\pi i \theta \sqrt{N}}
		- 2\pi i \theta \beta^{-1} \int_{\sqrt{N}}^N t^\beta e^{2\pi i \theta t} {\: \rm d}t
		+\calO\big(\sqrt{N}\big).
	\end{equation}
	By the Page's theorem, there is an absolute constant $c > 0$ such that for each $x \geq 2$,
	$1 \leq q \leq \exp\big(c\sqrt{\log x}\big)$, and $r \in A_q$,
	\[
		\bigg|
		\vartheta(x; q, r) - \frac{x}{\varphi(q)}
		\bigg|
		\lesssim
		x \exp\big(-c \sqrt{\log x} \big),
	\]
	if there is no exceptional character modulo $q$, and
	\[
		\bigg|
        \vartheta(x; q, r) 
		-
		\frac{x}{\varphi(q)}
		+
		\frac{\chi(r)}{\varphi(q)} \beta^{-1} x^\beta
        \bigg|
        \lesssim
        x \exp\big(-c \sqrt{\log x} \big),
	\]
	when there is an exceptional character $\chi$ modulo $q$, and $\beta$ is the concomitant zero. Therefore, by
	\eqref{eq:10} and \eqref{eq:11}, we obtain
	\begin{align*}
		&	
		\bigg|
		\sum_{\atop{p \in \PP_N}{p \equiv r \bmod q}}
        e^{2 \pi i \theta p} \log p
		-
		\frac{1}{\varphi(q)}
		\sum_{n = 1}^N
		e^{2\pi i \theta n}
		\bigg(1 - \chi(r) \frac{n^\beta - (n-1)^\beta}{\beta}\bigg)
		\bigg| \\
		&\qquad\qquad\lesssim
		\sqrt{N}
		+
		\bigg|
		\vartheta(N; q, r)
		-
        \frac{N}{\varphi(q)}
        +
        \frac{\chi(r)}{\varphi(q)} \beta^{-1} N^\beta
		\bigg|
		+
		\bigg|
		\vartheta(\sqrt{N}; q, r) 
		-
        \frac{\sqrt{N}}{\varphi(q)}
        +
        \frac{\chi(r)}{\varphi(q)} \beta^{-1} \sqrt{N^{\beta}}
		\bigg| \\
		&\qquad\qquad\phantom{\lesssim}
		+
		\abs{\theta}
		\int_{\sqrt{N}}^N 
		\bigg|\vartheta(t; q, r) - \frac{t}{\varphi(q)} + \frac{\chi(r)}{\varphi(q)} \beta^{-1} t^{\beta}\bigg|
		{\: \rm d}t 
		\\
		&\qquad\qquad\lesssim
		N \exp\big(-c\sqrt{\log N}\big) + Q N^{-1} \int_{\sqrt{N}}^N t \exp\big(-c\sqrt{\log t}\big) {\: \rm d}t,
	\end{align*}
	which is bounded by $N Q \exp\big(-c\sqrt{\log N}\big)$. Finally, by the prime number theorem
	\[
		\bigg| \frac{\vartheta(N) - N}{N} \bigg| \leq C \exp\big(-c\sqrt{\log N}\big),
	\]
	and the proposition follows.
\end{proof}
Next, we select $\eta : \RR \rightarrow \RR$, a smooth function such that $0 \leq \eta \leq 1$, and
\[
    \eta(\xi) =
    \begin{cases}
        1 & \text{if } \abs{\xi} \leq \tfrac{1}{4}, \\
        0 & \text{if } \abs{\xi} \geq \tfrac{1}{2}.
    \end{cases}
\]
We may assume that $\eta$ is a convolution of two smooth functions with supports contained in
$\big(-\tfrac{1}{2}, \tfrac{1}{2}\big)$. For $s \in \NN_0$, we set
\[
	\eta_s(\xi) = \eta\big(2^{4 s} \xi \big).
\]
We define a family of approximating multipliers, by the formula
\begin{equation}
	\label{eq:44}
	\nu_n^s(\xi) = \sum_{a/q \in \mathscr{R}_s}
	\widehat{L_{2^n}^{a, q}}(\xi - a/q)
	\eta_s\big(\xi - a/q\big)
\end{equation}
where
\begin{align*}
	\mathscr{R}_s 
	= \big\{
	a/q \in \QQ \cap (0, 1] : a \in A_q, \text{ and } 2^s \leq q < 2^{s+1}, q \text{ is square-free or } 
	4 \mid q \text{ and } q/4 \text{ is square-free}
	\big\},
\end{align*}
and $\scrR_0 = \{1\}$. We set $\nu_n = \sum_{s \geq 0} \nu_n^s$.
\begin{theorem}
	\label{thm:1}
	There are $C, c > 0$ such that for all $n \in \NN_0$ and $\xi \in \TT$,
	\[
		\big|
		\mathfrak{m}_{2^n}(\xi) - 
		\nu_n(\xi)
		\big|
		\leq
		C \exp\big(-c \sqrt{n}\big)
	\]
	where $\mathfrak{m}_N$ is defined by \eqref{eq:29}.
\end{theorem}
\begin{proof}
	Let
	\[
		Q_n = \exp\big(\tfrac{c}{2} \sqrt{n} \big)
	\]
	where the constant $c$ is determined in Proposition \ref{prop:1}. By the Dirichlet's principle, there are coprime
	integers $a$ and $q$, satisfying $1 \leq a \leq q \leq 2^n Q_n^{-1}$, and such that
	\[
		\bigg| \xi - \frac{a}{q} \bigg| \leq \frac{1}{q} 2^{-n} Q_n.
	\]
	Let us first consider the case when $1 \leq q \leq Q_n$. We select $s_1 \in \NN_0$ satisfying
	\[
        2^{s_1+1} < \frac{1}{2} 2^n Q_n^{-2} \leq 2^{s_1+2}.
    \]
	For $s \leq s_1$ and $a'/q' \in \mathscr{R}_s$, with $a'/q' \neq a/q$, we have
	\[
		\bigg| 	\xi - \frac{a'}{q'} \bigg|
		\geq 
		\frac{1}{q q'}
		-
		\bigg|\xi - \frac{a}{q} \bigg|
		\geq
		Q_n^{-1} 2^{-s_1-1}
		- 2^{-n} Q_n \geq 2^{-n} Q_n.
	\]
	Therefore, by \eqref{eq:15} and \eqref{eq:14},
	\[
		\bigg|
		\widehat{L_{2^n}^{a', q'}} (\xi - a'/q') \eta_s(\xi - a'/q')
		\bigg|
		\lesssim
		2^{-\frac{s}{4}}
		\big|2^n (\xi - a'/q') \big|^{-1} \leq 2^{-\frac{s}{4}} Q_n^{-1},
	\]
	which implies that
	\begin{align*}
		\bigg|
		\sum_{s = 0}^{s_1} 
		\sum_{\atop{a'/q' \in \scrR_s}{a'/q' \neq a/q}} 
		\widehat{L^{a', q'}_{2^n}}(\xi - a'/q') 
		\eta_s(\xi - a'/q')
		\bigg|
		\lesssim
		Q_n^{-1} \sum_{s \geq 0} 2^{-\frac{s}{4}}.
	\end{align*}
	For $s > s_1$, by \eqref{eq:15} we obtain
	\[
		\bigg|
		\sum_{s > s_1}
		\sum_{\atop{a'/q' \in \scrR_s}{a'/q' \neq a/q}}
		\widehat{L^{a', q'}_{2^n}}(\xi - a'/q') \eta_s(\xi-a'/q')
		\bigg|
		\lesssim
		\sum_{s > s_1}
		2^{-\frac{s}{4}}
		\lesssim 
		\big(2^n Q_n^{-2}\big)^{-\frac{1}{4}} \lesssim Q_n^{-1}.
	\]
	If $q$ is square-free or $4 \mid q$ and $q/4$ is square-free then there is $s_0 \in \NN_0$ such that
	$a/q \in \mathscr{R}_{s_0}$, thus
	\[
		Q_n \geq 2^{s_0}.
	\]
	By Proposition \ref{prop:1}, 
	\[
		\bigg|
		\mathfrak{m}_{2^n}(\xi) - \widehat{L^{a, q}_{2^n}} (\xi - a/q) \eta_{s_0}(\xi - a/q)
		\bigg|
		\lesssim
		\big|2^n (\xi - a/q)\big|^{-1} \big(1 - \eta_{s_0}(\xi - a/q)\big) 
		+ Q_n^{-1}.
	\]
	Since $1 - \eta_{s_0}(\xi - a/q) > 0$, whenever
	\[
		\bigg|
		\xi - \frac{a}{q}
		\bigg|
		\geq \frac{1}{4} 2^{-4 s_0} \gtrsim Q_n^{-4},
	\]
	we obtain
	\[
		\bigg|
		\mathfrak{m}_{2^n}(\xi)
		-
		\widehat{L^{a, q}_{2^n}}(\xi - a/q) \eta_{s_0}(\xi - a/q)
        \bigg|
		\lesssim
		2^{-n} Q_n^4 + Q_n^{-1}
		\lesssim Q_n^{-1}.
	\]
	Finally, if $q$ and $q/4$ are not square-free then by Proposition \ref{prop:1},
	\[
		\bigg|
        \mathfrak{m}_{2^n}(\xi)
        -
        \widehat{L^{a, q}_{2^n}}(\xi - a/q) \eta_{s_0}(\xi - a/q)
        \bigg|
		=
        \big|
        \mathfrak{m}_{2^n}(\xi)
        \big| \lesssim Q_n^{-1}.
    \]
	It remains to deal with $Q_n \leq q \leq 2^n Q_n^{-1}$. By the Vinogradov's inequality
	(see \cite[Theorem 1, Chapter IX]{vin} or \cite[Theorem 8.5]{nat}), we get
	\[
		\big|
		\mathfrak{m}_{2^n}(\xi) 
		\big| \lesssim 
		n^4 \Big(q^{-\frac{1}{2}} + 2^{-\frac{1}{2}n} q^{\frac{1}{2}} + 2^{-\frac{1}{5}n}\Big)
		\lesssim	
		n^4 Q_n^{-\frac{1}{2}}.
	\]
	Next, we show that
	\[
		\Big|
        \sum_{s \geq 0} \sum_{a'/q' \in \scrR_s} \widehat{L^{a', q'}_{2^n}}(\xi - a'/q') \eta_s(\xi - a'/q')
        \Big|
        \lesssim
        Q_n^{-\frac{1}{8}}.
	\]
	Select $s_2 \in \NN_0$ such that
	\begin{equation}
		\label{eq:38}
		2^{s_2+1} \leq Q_n^\frac{1}{2} \leq 2^{s_2+2}. 
	\end{equation}
	For $s \leq s_2$, if $a'/q' \in \mathscr{R}_s$, then $1 \leq q' \leq Q_n^\frac{1}{2}$, and hence
	\[
		\bigg|\xi - \frac{a'}{q'}\bigg| \geq \frac{1}{q'} 2^{-n} Q_n \geq 2^{-n} Q_n^{\frac{1}{2}}.
	\]
	Therefore, by \eqref{eq:15} and \eqref{eq:14},
	\[
		\bigg|
        \widehat{L^{a', q'}_{2^n}}(\xi - a'/q') \eta_s(\xi - a'/q')
        \bigg|
		\lesssim
		2^{-\frac{s}{4}} Q_n^{-\frac{1}{2}},
	\]
	which entails that
	\[
		\Big|
		\sum_{s=0}^{s_2} \sum_{a'/q' \in \scrR_s} \widehat{L^{a', q'}_{2^n}}(\xi - a'/q') \eta_s(\xi - a'/q')
		\Big|
		\lesssim
		Q_n^{-\frac{1}{2}} \sum_{s \geq 0} 2^{-\frac{s}{4}}.
	\]
	If $s > s_2$, then by \eqref{eq:15}, we get
	\[
		\bigg|
        \widehat{L^{a', q'}_{2^n}}(\xi - a'/q') \eta_s(\xi - a'/q')
        \bigg|
		\lesssim
		2^{-\frac{s}{4}},
	\]
	hence by \eqref{eq:38},
	\[
		\Big|
        \sum_{s > s_2} \sum_{a'/q' \in \scrR_s} \widehat{L^{a', q'}_{2^n}}(\xi - a'/q') \eta_s(\xi - a'/q')
        \Big|
        \lesssim
		\sum_{s > s_2} 2^{-\frac{s}{4}} \lesssim Q_n^{-\frac{1}{8}},
	\]
	and the theorem follows.
\end{proof}

\section{Equidistribution of weak $\ell^1$ norms}
\label{sec:5}
In this section we prove that the maximal function associated with kernels $(M^\beta_{2^n} : n \in \NN_0)$
has weak $\ell^1(\ZZ)$-norm equidistributed in residue classes. Before embarking on the proof, let us recall two lemmas
essential for the argument.
\begin{lemma}{\cite[Lemma 1]{mt2}}
	\label{lem:1}
	There is $C > 0$ such that for all $s \in \NN$ and $u \in \RR$, 
	\begin{align*}
		\bigg\|\int_{-\frac{1}{2}}^{\frac{1}{2}} 
		e^{2 \pi i \xi x} \eta_s(\xi) {\: \rm d}\xi \bigg\|_{\ell^1(x)} &\leq C,\\
		\bigg\|\int_{-\frac{1}{2}}^{\frac{1}{2}} 
		e^{2 \pi i \xi x} \big(1 - e^{2 \pi i \xi u} \big) \eta_s(\xi) 
		{\: \rm d}\xi \bigg\|_{\ell^1(x)} &\leq C \abs{u} 2^{-4 s}.
	\end{align*}
\end{lemma}
\begin{lemma}{\cite[Lemma 2]{mt2}}
	\label{lem:2}
	For all $p \geq 1$, any $1 \leq Q \leq 2^{2 s}$ with $s \in \NN$, 
	$r \in \{1, \ldots, Q\}$, and any finitely supported function $f: \ZZ \rightarrow \CC$, 
	\[
		\big\|
		\calF^{-1}\big(\eta_s \hat{f} \big)(Q x + r)
		\big\|_{\ell^p(x)}
		\simeq
		Q^{-\frac{1}{p}} \big\|
        \calF^{-1}\big(\eta_s \hat{f} \big)
        \big\|_{\ell^p}.
	\]
\end{lemma}
The following theorem is the main result of this section.
\begin{theorem}
	\label{thm:6}
	There is $C > 0$ such that for any $1 \leq Q \leq 2^{2 s}$ with $s \in \NN$, $r \in \{1, \ldots, Q\}$,
	$\frac{1}{2} \leq \beta \leq 1$, and any finitely supported function $f: \ZZ \rightarrow \CC$, 
	\[
		\sup_{\lambda > 0}{
		\lambda \cdot 
		\Big| \Big\{
		x \in \ZZ :
		\sup_{n \in \NN_0} \big|
		M_{2^n}^\beta *\calF^{-1}\big(\eta_s \hat{f}\big) (Q x + r)
		\big|
		> \lambda
		\Big\}\Big|}
		\leq
		C 
		\big\|\calF^{-1} \big(\eta_s \hat{f}\big)(Qx+r)\big\|_{\ell^1(x)}.
	\]
\end{theorem}
\begin{proof}
	Observe that, by the mean value theorem, for $x \in \NN$,
	\[
		\frac{x^\beta - (x-1)^\beta}{\beta} \leq x^{\beta-1} \leq 1,
	\]
	thus
	\[
		M^\beta_N(x) \leq M_N(x).
	\]
	In particular, by the Hardy--Littlewood maximal theorem, there is $C > 0$ such that for all
	$\frac{1}{2} \leq \beta \leq 1$, and any $f \in \ell^1(\ZZ)$,
	\begin{equation}
		\label{eq:16}
		\sup_{\lambda > 0} {
		\lambda \cdot \Big|
		\Big\{x \in \ZZ : 
		\sup_{n \in \NN_0}
		\big| M_{2^n}^\beta * f (x) \big| > \lambda
		\Big\}
		\Big|}
		\leq
		C
		\|f\|_{\ell^1}. 
	\end{equation}
	For $r \in \{1, \ldots, Q\}$ and $\lambda > 0$, we set
	\[
		J_r(\lambda) = \Big|\Big\{
        x \in \ZZ :
        \sup_{n \in \NN_0}
		\big|
        M_{2^n}^\beta * \calF^{-1}\big( \eta_{s} \hat{f} \big)(Q x + r)\big| > \lambda
        \Big\}\Big|.
	\]
	Then, by \eqref{eq:16}, we have
	\begin{align}
		\nonumber
		J_1(\lambda) + \ldots + J_Q(\lambda) 
		&= \Big|\Big\{ x \in \ZZ :
        \sup_{n \in \NN_0}  \big|
		M_{2^n}^\beta * \calF^{-1}\big( \eta_{s} \hat{f} \big) (x) \big| > \lambda
        \Big\}\Big| \\
		\label{eq:33}
		&\leq
		C \lambda^{-1} \big\|\calF^{-1}\big(\eta_{s} \hat{f}\big)\big\|_{\ell^1}.
	\end{align}
	Moreover, for any $r, r' \in \{1, \ldots, Q\}$, we have
	\begin{align*}
		&\bigg|\bigg\{x \in \ZZ  :
		\sup_{n \in \NN_0}
		\bigg|
		\int_0^1 e^{2\pi i \xi (Qx +r)} \Big(1 - e^{2\pi i \xi (r' - r)} \Big)
		\widehat{M_{2^n}^\beta}(\xi) \eta_{s}(\xi) \hat{f}(\xi)
		{\: \rm d}\xi
		\bigg|
		> \tfrac{1}{2} \lambda
		\bigg\}\bigg| \\
		&\qquad\qquad\leq
		C \lambda^{-1}
		\bigg\|
		\int_0^1 e^{2\pi i \xi x}
		\Big( 1 - e^{2\pi i \xi (r' - r)} \Big) \eta_{s}(\xi) \hat{f}(\xi) {\: \rm d}\xi
		\bigg\|_{\ell^1(x)}.
	\end{align*}
	Since $\eta_{s} = \eta_{s} \eta_{s-1}$, by Young's convolution inequality and Lemma \ref{lem:1}, we
	obtain
	\begin{align*}
		&\bigg\|
        \int_0^1 e^{2\pi i \xi x}
        \Big(1 - e^{2\pi i \xi (r' - r)} \Big) \eta_{s}(\xi) \hat{f}(\xi) {\: \rm d}\xi
        \bigg\|_{\ell^1(x)}\\
		&\qquad\qquad\leq
		\bigg\|
        \int_0^1 e^{2\pi i \xi x}
        \Big( 1 - e^{2\pi i \xi (r' - r)} \Big) \eta_{s-1}(\xi)
        \bigg\|_{\ell^1(x)}
		\big\|\calF^{-1} \big(\eta_{s} \hat{f} \big) \big\|_{\ell^1} \\
		&\qquad\qquad\leq
		C Q 2^{-4s} \big\|\calF^{-1}\big(\eta_{s} \hat{f} \big)\big\|_{\ell^1}.
	\end{align*}
	Thus
	\[
		J_r(\lambda) \leq J_{r'}(\lambda/2) + C \lambda^{-1} Q 2^{-4 s} 
		\big\|\calF^{-1} \big( \eta_{s} \hat{f} \big)\big\|_{\ell^1},
	\]
	which together with \eqref{eq:33} imply that
	\begin{align*}
		Q J_r(\lambda) &\leq J_1(\lambda/2) + \ldots + J_Q(\lambda/2) + C \lambda^{-1} Q^2 2^{-4 s} 
		\big\|\calF^{-1} \big( \eta_{s} \hat{f} \big)\big\|_{\ell^1} \\
		&\lesssim \lambda^{-1} \Big(1+Q^2 2^{-4 s}\Big) \big\|\calF^{-1} \big( \eta_{s} \hat{f} \big)\big\|_{\ell^1} \\
		&\lesssim \lambda^{-1} \big\|\calF^{-1} \big( \eta_{s} \hat{f} \big)\big\|_{\ell^1},
	\end{align*}
	where the last inequality is a consequence of $1 \leq Q \leq 2^{2s}$. Therefore, in view of Lemma \ref{lem:2},
	we immediately get
	\[
		Q J_r(\lambda) 
		\lesssim
		\lambda^{-1} 
		\big\|\calF^{-1}\big(\eta_{s} \hat{f} \big)(Qx+r)\big\|_{\ell^1(x)},
	\]
	which is the desired conclusion.
\end{proof}
Essentially the same reasoning as in the proof of Theorem \ref{thm:6} leads to the following theorem.
\begin{theorem}
	\label{thm:9}
	There is $C  > 0$ such that for all $1 \leq Q \leq 2^{2s}$ with $s \in \NN$, $r \in \{1, \ldots, Q\}$,
	$\frac{1}{2} \leq \beta \leq 1$, and any finitely supported function $f: \ZZ \rightarrow \CC$,
	\[
		\Big\|
		\sup_{n \in \NN_0}
		\Big|
		\calF^{-1}\big(\widehat{M_{2^n}^\beta} \eta_s \hat{f} \big)(Qx+r)
		\Big|
		\Big\|_{\ell^2(x)}
		\leq
		C \big\|\calF^{-1}\big(\eta_s \hat{f}\big)(Qx+r)f\big\|_{\ell^2(x)}.
	\]
\end{theorem}

\section{$\ell^2$ theory}
\label{sec:2}
We are now in the position to prove $\ell^2(\ZZ)$ boundedness of the maximal function associated to the multipliers
$(\nu_n^s : n \in \NN)$.
\begin{theorem}
	\label{thm:2}
	For each $\epsilon > 0$ there is $C > 0$ such that for all $s \in \NN_0$, and any finitely supported function 
	$f : \ZZ \rightarrow \CC$,
	\[
		\Big\|
		\sup_{n \in \NN}
		\big|
		\calF^{-1}\big(\nu_n^s \hat{f} \big)
		\big|
		\Big\|_{\ell^2}
		\leq
		C 2^{-s(\frac{1}{2}-\epsilon)} \|f\|_{\ell^2}.
	\]
\end{theorem}
\begin{proof}
	We divide the supremum into two parts: $0 \leq n < 2^{s+4}$ and $2^{s+4} \leq n$. Then the following holds true.
	\begin{claim}
		For each $\epsilon > 0$ there is $C > 0$ such that for all $s \in \NN_0$, and any finitely supported function
		$f: \ZZ \rightarrow \CC$,
		\begin{equation}
			\label{eq:26}
			\Big\|
			\sup_{0 \leq n \leq 2^{s+4}}
			\big|
			\calF^{-1}\big(
			\nu_{n}^s \hat{f} 
			\big)
			\big|
			\Big\|_{\ell^2}
			\leq
			C (s+1) 2^{-s(\frac{1}{2}-\epsilon)} \|f\|_{\ell^2}.
		\end{equation}
	\end{claim}
	For the proof, we apply \cite[Lemma 1]{mt3} to write
	\begin{equation}
		\label{eq:25}
		\sup_{0 \leq n < 2^{s+4}}
		\Big|
	    \calF^{-1}\big(
        \nu_n^s \hat{f}
	    \big)
		\Big|
		\leq
		\big|\calF^{-1}\big(\nu_0^s \hat{f} \big)\big|
		+
		\sqrt{2}
		\sum_{i = 0}^{s+4}
		\Big(
		\sum_{j = 0}^{2^{s+4-i}-1} 
		\big|
		\calF^{-1}\big((\nu_{(j+1) 2^i}^s - \nu_{j2^i}^s) \hat{f}
		\big|^2
		\Big)^{\frac{1}{2}}.
	\end{equation}
	Let us fix $i \in \{0, \ldots, s\}$. Then by the Plancherel's theorem we get
	\begin{align*}
		&\sum_{j = 0}^{2^{s+4-i}-1}
        \Big\|
        \calF^{-1}\big((\nu_{(j+1) 2^i}^s - \nu_{j2^i}^s) \hat{f}\big)
        \Big\|_{\ell^2}^2\\ 
		&\qquad\qquad=
		\sum_{j = 0}^{2^{s+4-i}-1}
		\sum_{a/q \in \scrR_s}
		\int_0^1
		\Big|\sum_{m \in I_j^i} \widehat{L^{a, q}_{2^m}} (\xi - a/q) - \widehat{L^{a, q}_{2^{m-1}}} (\xi - a/q) \Big|^2
		\eta_s(\xi - a/q)^2 
		\abs{\hat{f}(\xi)}^2 {\: \rm d}\xi
	\end{align*}
	where $I_j^i = \big\{j2^i+1, j2^i + 2, \ldots, (j+1)2^i\big\}$. By \eqref{eq:15}, we obtain
	\begin{align*}
		&
		\sum_{j = 0}^{2^{s+4-i}-1}
        \sum_{a/q \in \scrR_s}
        \int_0^1
        \Big|\sum_{m \in I_j^i} \widehat{L^{a, q}_{2^m}} (\xi - a/q) - \widehat{L^{a, q}_{2^{m-1}}} (\xi - a/q) \Big|^2
		\eta_s(\xi - a/q)^2
        \abs{\hat{f}(\xi)}^2 {\: \rm d}\xi \\
		&\qquad\qquad
		\lesssim
		2^{-s(1-\epsilon)}
		\sum_{a/q \in \scrR_s}
		\sum_{j = 0}^{2^{s+4-i}-1}
		\sum_{m, m' \in I_j^i}
		\int_0^1
		\Delta_m^q(\xi - a/q) \cdot \Delta^q_{m'}(\xi - a/q) \cdot 
		\eta_s(\xi - a/q)^2 \abs{\hat{f}(\xi)}^2 {\: \rm d}\xi,
	\end{align*}
	where $\Delta^q_m = \big|\widehat{M_{2^m}} - \widehat{M_{2^{m-1}}}\big| + 
	\big|\widehat{M^{\beta_q}_{2^m}} - \widehat{M^{\beta_q}_{2^{m-1}}}\big|$. In view of \eqref{eq:24}, we have
	\[
		\sum_{n \in \NN_0} \Delta^q_m(\xi) \lesssim
		\sum_{n \in \NN_0} \min\big\{(2^n \abs{\xi})^{-1}, 2^n \abs{\xi}\big\} + (1-\beta_q) 2^{-n(1-\beta_q)}
		\lesssim
		1,
	\]
	uniformly with respect to $\xi \in \TT$, $q \in \NN$, and $\frac{1}{2} \leq \beta_q \leq 1$. Since
	supports of $\eta_s(\cdot - a/q)$ are disjoint while $a/q$ varies over $\scrR_s$, we obtain
	\begin{align*}
		\sum_{j = 0}^{2^{s+4-i}-1}
        \Big\|
        \calF^{-1}\big((\nu_{(j+1) 2^i}^s - \nu_{j2^i}^s) \hat{f}\big)
        \Big\|_{\ell^2}^2
		&\lesssim
		2^{-s(1-\epsilon)}
		\sum_{a/q \in \scrR_s}
		\int_0^1 \eta_s(\xi - a/q)^2 \abs{\hat{f}(\xi)}^2 {\: \rm d}\xi\\
		&\lesssim
		2^{-s(1-\epsilon)}
        \|f\|_{\ell^2}^2,
	\end{align*}
	which together with \eqref{eq:25} imply \eqref{eq:26}.

	It remains now to treat supremum over $n \geq 2^{s+4}$. For each $\frac{1}{2} \leq \beta < 1$ we set
	\[
		\scrR_s^\beta = \big\{a/q \in \scrR_s :  \beta_q = \beta \big\}.
	\]
	and $\scrR_s^1 = \scrR_s$. In view of the Landau's theorem \cite[Corollary 11.9]{mv}, there are $\calO(\log s)$ 
	distinct $\beta$'s. Therefore, it suffices to show the following claim.
	\begin{claim}
		For each $\epsilon > 0$ there is $C > 0$ such that for all $s \in \NN_0$, $\frac{1}{2} \leq \beta \leq 1$,
		any finitely supported function $f: \ZZ \rightarrow \CC$,
		\begin{equation}
			\label{eq:28}
			\Big\|
            \sup_{2^{s+4} \leq n}
            \Big|
			\sum_{a/q \in \scrR_s^\beta}
			G(\chi_q, a)
            \calF^{-1}\big(\widehat{M_{2^n}^{\beta}}(\cdot - a/q) \eta_s(\cdot - a/q) \hat{f}
            \big)
            \Big|
            \Big\|_{\ell^2}
            \leq
			C
			2^{-s(\frac{1}{2}-\epsilon)} \|f\|_{\ell^2}.
		\end{equation}
	\end{claim}
	Let us fix $\frac{1}{2} \leq \beta \leq 1$. We define
	\[
		I(x, y) = \sup_{2^{s+4} \leq n} 
		\Big|
		\sum_{a/q \in \scrR_s^\beta}
		G(\chi_q, a) e^{2\pi i x a/q} 
		\calF^{-1}\big(\widehat{M_{2^n}^{\beta}} \eta_s \hat{f}(\cdot + a/q)\big)(y)
		\Big|,
	\]
	and
	\[
		J(x, y) = \sum_{a/q \in \scrR_s^\beta} G(\chi_q, a) e^{2 \pi i x a/q} 
		\calF^{-1}\big(\eta_s \hat{f}(\cdot + a/q)\big)(y).
	\]
	Observe that the functions $x \mapsto I(x, y)$ and $x \mapsto J(x, y)$ are $Q_s$ periodic where
	\[
		Q_s = 4 \prod_{p \in \PP_{2^{s+1}}} p \lesssim e^{2^{s+2}}.
	\]
	By the Plancherel's theorem, for $u \in \ZZ_{Q_s}$, we have
	\begin{align*}
		&
		\big\|
		\calF^{-1}\big(\widehat{M^{\beta}_{2^n}} \eta_s \hat{f}(\cdot + a/q) \big)(x+u)
		-
		\calF^{-1}\big(\widehat{M^{\beta}_{2^n}} \eta_s \hat{f}(\cdot + a/q) \big)(x)
		\big\|_{\ell^2(x)}\\
		&\qquad\qquad=
		\bigg\|
		\big(1 - e^{2\pi i \xi u}\big) \widehat{M^{\beta}_{2^n}}(\xi) \eta_s(\xi) \hat{f}(\xi+a/q) 
		\bigg\|_{L^2({\rm d}\xi)} \\
		&\qquad\qquad\lesssim
		2^{-n} \abs{u} \cdot \big\|\eta_s \hat{f}(\cdot + a/q)\big\|_{L^2},
	\end{align*}
	because by \eqref{eq:14},
	\[
		\sup_{\xi \in \TT}{\abs{\xi} \cdot \abs{\widehat{M^{\beta}}_{2^n}(\xi)}} \lesssim 2^{-n}.
	\]
	Therefore, by the triangle inequality
	\[
		\Big|
		\big\| I(x, x+u) \big\|_{\ell^2(x)} - \big\|I(x, x) \big\|_{\ell^2(x)}
		\Big|
		\lesssim
		Q_s \sum_{n \geq 2^{s+4}} 2^{-n} \sum_{a/q \in \scrR_s} \big\|\eta_s \hat{f}(\cdot + a/q)\big\|_{L^2}.
	\]
	Since $\scrR_s$ contains at most $2^{2(s+1)}$ rational numbers, by the Cauchy--Schwarz inequality we get
	\[
		\sum_{a/q \in \scrR_s} \big\|\eta_s \hat{f}(\cdot + a/q)\big\|_{L^2}
		\leq
		2^{s+1} \|f\|_{\ell^2}.
	\]
	Observe that
	\[
		Q_s \cdot 2^{-2^{s+4}} \cdot 2^{s+1} \leq 2^{2^{s+3} - 2^{s+4} + s + 1} \leq 2^{-s},
	\]
	thus
	\[
		\big\| I(x, x) \big\|_{\ell^2(x)}
		\lesssim
		\big\|I(x, x+u) \big\|_{\ell^2(x)}
		+
		2^{-s} \|f\|_{\ell^2}.
	\]
	Hence,
	\begin{equation}
		\label{eq:27}
		\big\|I(x, x) \big\|_{\ell^2(x)}^2
		\lesssim
		\frac{1}{Q_s} \sum_{u = 1}^{Q_s} \big\|I(x, x+u) \big\|_{\ell^2(x)}^2 + 2^{-2s} \|f\|_{\ell^2}^2.
	\end{equation}
	Now, by multiple change of variables and periodicity we get
	\[
		\sum_{u = 1}^{Q_s} \big\|I(x, x+u) \big\|_{\ell^2(x)}^2
		=
		\sum_{u = 1}^{Q_s} \sum_{x \in \ZZ} I(x-u, x)^2
		=
		\sum_{x \in \ZZ} \sum_{u = 1}^{Q_s} I(u, x)^2
		=
		\sum_{u = 1}^{Q_s} \big\|I(u, x) \big\|_{\ell^2(x)}^2.
	\]
	Using Theorem \ref{thm:9}, we can estimate
	\begin{align*}
		\big\|I(u, x) \big\|_{\ell^2(x)}
		=
		\Big\|
		\sup_{2^{s+4} \leq n} \Big|
		\calF^{-1}\big(\widehat{M^\beta_{2^n}} \eta_s J(u, \cdot) \big)
		\Big|
		\Big\|_{\ell^2}
		\lesssim
		\big\|J(u, x)\big\|_{\ell^2(x)}.
	\end{align*}
	Notice that
	\[
		\sum_{u = 1}^{Q_s} \big\|J(u, x) \big\|_{\ell^2(x)}^2 
		=
		\sum_{x \in \ZZ} \sum_{u = 1}^{Q_s} J(u, x)^2
		=
		\sum_{u = 1}^{Q_s} \sum_{x \in \ZZ} J(x-u, x)^2
		=
		\sum_{u = 1}^{Q_s} \big\|J(x, x+u) \big\|_{\ell^2(x)}^2.
	\]
	Since supports of $\eta_s(\cdot - a/q)$ are disjoint while $a/q$ varies over $\scrR_s$, by \eqref{eq:15}
	we get
	\begin{align*}
		\big\|J(x, x+u) \big\|_{\ell^2(x)}^2
		&=
		\int_0^1
		\Big|
		\sum_{a/q \in \scrR_s^\beta} G(\chi_q, a) e^{2\pi i \xi u a/q} \eta_s(\xi - a/q)
		\Big|^2 \abs{\hat{f}(\xi)}^2 {\:\rm d}\xi \\
		&\lesssim
		2^{-s(1-\epsilon)} \|f\|_{\ell^2}^2.
	\end{align*}
	Therefore,
	\[
		\sum_{u = 1}^{Q_s} \big\|I(x, x+u) \big\|_{\ell^2(x)}^2 \lesssim 2^{-s(1-\epsilon)} Q_s \|f\|_{\ell^2}^2,
	\]
	which together with \eqref{eq:27} imply \eqref{eq:28} and the theorem follows.
\end{proof}
Given $t > 0$ and $n > t$, we define the multiplier
\begin{align*}
	\Pi_n^t(\xi) 
	&= \sum_{0 \leq s \leq \sqrt{t}} \nu_n^s(\xi) \\
	&= 
	\sum_{0 \leq s \leq \sqrt{t}}
	\sum_{a/q \in \scrR_s}
	\widehat{L^{a, q}_{2^n}}(\xi - a/q)
	\eta_s(\xi - a/q).
\end{align*}
\begin{corollary}
	\label{cor:3}
	There are $C, c > 0$ such that for each $t > 0$, and any finitely supported function $f \in \ZZ \rightarrow \CC$,
	\[
		\Big\|
		\sup_{t \leq n}
		\Big|
		\calM_{2^n} f - \calF^{-1}\big(\Pi_n^t \hat{f} \big)
		\Big|
		\Big\|_{\ell^2}
		\leq
		C \exp\big(-c \sqrt{t}\big) \|f\|_{\ell^2}.
	\]
\end{corollary}
\begin{proof}
	Since
	\begin{align*}
    	\mathfrak{m}_{2^n} - \Pi_n^t
		&= \big(\mathfrak{m}_{2^n} - \nu_n\big)
		+ \sum_{s > \sqrt{t}} \nu_n^s,
	\end{align*}
	our assertion follows from Theorem \ref{thm:1} and Theorem \ref{thm:2}. Indeed, by the Plancherel's theorem 
	and Theorem \ref{thm:1} we get
	\[
		\Big\|\sup_{t \leq n} \Big| \calF^{-1}\big((\mathfrak{m}_{2^n} - \nu_n)\hat{f}\big) \Big|\Big\|_{\ell^2}
		\lesssim
		\Big(\sum_{n \geq t} \exp\big(-2 c \sqrt{n} \big)\Big)^\frac{1}{2} \|f\|_{\ell^2}.
	\]
	On the other hand, by Theorem \ref{thm:2},
	\[
		\Big\|\sup_{t \leq n} \Big| \sum_{s > \sqrt{t}} \calF^{-1}\big( \nu_n^s \hat{f} \big) \Big|\Big\|_{\ell^2}
		\lesssim
		\sum_{s > \sqrt{t}} 2^{-\frac{s}{4}} \|f\|_{\ell^2},
	\]
	which concludes the proof.
\end{proof}

\section{Weak type estimates}
\label{sec:1}
In this section we investigate the weak type estimates for the multipliers $\big(\Pi_n^t : n \geq t\big)$. Then together
with results from Section \ref{sec:2} we deduce Theorem \ref{main_thm:2}.
\begin{theorem}
	\label{thm:3}
	There is $C > 0$ such that for all $t > 0$ and any finitely supported function $f: \ZZ \rightarrow \CC$,
	\[
		\sup_{\lambda > 0}{
		\lambda \cdot
		\Big|\Big\{
		x \in \ZZ : 
		\sup_{t \leq n}
		\big|
		\calF^{-1}
		\big(
		\Pi_n^t \hat{f}
		\big)(x)
		\big| > \lambda
		\Big\}\Big|}
		\leq
		C
		t  
		\|f\|_{\ell^1}.
	\]
\end{theorem}
\begin{proof}
	Let us fix $2^s \leq q < 2^{s+1}$ for some  $1 \leq s \leq \sqrt{t}$. Let $\frac{1}{2} \leq \beta \leq 1$. Suppose
	that $\chi$ is a quadratic Dirichlet character modulo $q$ induced by $\chi^\star$ having the conductor $q_0$. We claim
	that the following holds true.
	\begin{claim}
		\label{clm:1}
		There is $C > 0$ such that for any finitely supported function $f: \ZZ \rightarrow \CC$,
		\begin{equation}
			\label{eq:36}
			\sup_{\lambda > 0}{
			\lambda \cdot \Big|\Big\{
			\sup_{t \leq n} 
			\Big|
			\sum_{a \in A_q}
			G(\chi, a)
    	    \calF^{-1}
    	    \big(\widehat{M^{\beta}_{2^n}}(\cdot - a/q) \eta_s(\cdot - a/q) 
			\hat{f}\big)
			\Big| > \lambda
			\Big\}\Big|}
			\leq
			C
			\frac{1}{\varphi(q)}
			\| f \|_{\ell^1}.
		\end{equation}
		The constant $C$ is independent of $q$, $\beta$ and $\chi$. 
	\end{claim}
	Let us first see that from Claim \ref{clm:1}, we can deduce the theorem. Indeed, from \eqref{eq:36} we easily
	get
	\[
		\Big|\Big\{
        \sup_{t \leq n}
        \Big|
        \sum_{a \in A_q}
        \calF^{-1}
        \big(\widehat{L^{a, q}_{2^n}}(\cdot - a/q) \eta_s(\cdot - a/q)
        \hat{f}\big)
        \Big| > \lambda
        \Big\}\Big|
        \leq
        C
		\frac{1}{\lambda \varphi(q)}
		\| f \|_{\ell^1}.
	\]
	Recall that (see e.g. \cite{sita}),
	\[
		\sum_{1 \leq q < 2^{\sqrt{t}}} \frac{1}{\varphi(q)} \simeq \sqrt{t},
	\]
	thus
	\[
		\Phi(t) = \sum_{1 \leq q < 2^{\sqrt{t}}} \frac{1+\log q}{\varphi(q)}
		\lesssim t.
	\]
	Hence, by log-convexity of $\ell^{1,\infty}(\ZZ)$, (see \cite{kal, stwa}) we obtain
	\begin{align*}
		&\Big|\Big\{
		x \in \ZZ : 
		\sup_{t \leq n}
		\big|
		\calF^{-1}
		\big(
		\Pi_n^t \hat{f}
		\big)(x)
		\big| > \lambda
		\Big\}\Big| \\
		&=
		\Big|\Big\{
        \sup_{t \leq n}
        \Big|
		\sum_{0 \leq s \leq \sqrt{t}}
		\sum_{q = 2^s}^{2^{s+1}-1}
		\sum_{a \in A_q}
        \calF^{-1}
        \big(
		\widehat{L^{a, q}_{2^n}}(\cdot - a/q) \eta_s(\cdot - a/q) \hat{f}
        \big)
        \Big| > \lambda
        \Big\}\Big| \\
		&\qquad\qquad\leq
		\Big|\Big\{
		\sum_{0 \leq s \leq \sqrt{t}}
		\sum_{q = 2^s}^{2^{s+1}-1}
		\frac{1}{\varphi(q)} \varphi(q)
        \sup_{t \leq n}
        \Big|
		\sum_{a \in A_q}
        \calF^{-1}
        \big(
        \widehat{L^{a, q}_{2^n}}(\cdot - a/q) \eta_s(\cdot - a/q) \hat{f}
        \big)
        \Big| > \lambda
        \Big\}\Big|\\
		&\qquad\qquad\lesssim
		\lambda^{-1} \Phi(t)
		\|f\|_{\ell^1},
	\end{align*}
	which is bounded by $C \lambda^{-1} t \|f\|_{\ell^1}$.

	What is left now is to prove Claim \ref{clm:1}. Let $r \in \{1, \ldots, q\}$. For $x \equiv r \bmod q$, we have
	\begin{align*}
		\sum_{a \in A_q}
		G(\chi, a)
        \calF^{-1}
        \big(\widehat{M^{\beta}_{2^n}}(\cdot - a/q) \eta_s(\cdot - a/q)
        \hat{f}\big)(x)
		&=
		\sum_{a \in A_q}
		G(\chi, a) e^{2\pi i r a /q}
		\calF^{-1}
        \big(\widehat{M^{\beta}_{2^n}} \eta_s \hat{f}(\cdot + a/q) \big)(x) \\
		&=
		\calF^{-1}\big(\widehat{M^\beta_{2^n}} \eta_s F_q(\cdot; r)\big)(x),
	\end{align*}
	where
	\[
		F_q(\xi; r) = \sum_{a \in A_q} G(\chi, a) \hat{f}(\xi + a/q) e^{2\pi i r a / q}.
	\]
	Hence, by Theorem \ref{thm:6}, we obtain
	\begin{align*}
		&\Big|\Big\{
		x \in \ZZ : \sup_{t \leq n}
	    \Big|
        \sum_{a \in A_q}
        G(\chi, a)
        \calF^{-1}
        \big(\widehat{M^{\beta}_{2^n}}(\cdot - a/q) \eta_s(\cdot - a/q)
        \hat{f}\big)(x)
        \Big| > \lambda
        \Big\}\Big| \\
		&\qquad\qquad=
		\sum_{r = 1}^q
		\big|\big\{
        x \in \ZZ : \sup_{t \leq n}
        \big|
        \calF^{-1}
        \big(\widehat{M^{\beta}_{2^n}} \eta_s F_q(\cdot; r)
		\big)(q x + r)
        \big| > \lambda
        \big\}\big| \\
		&\qquad\qquad\lesssim
		\sum_{r = 1}^q
		\lambda^{-1}
		\big\|\calF^{-1}\big(\eta_s F_q(\cdot ; r) \big)(qx+r)\big\|_{\ell^1(x)}.
	\end{align*}
	Next, by Young's convolution inequality we get
	\begin{align*}
		\sum_{r = 1}^q
        \big\|\calF^{-1}\big(\eta_s F_q(\cdot ; r) \big)(qx+r)\big\|_{\ell^1(x)}
		&=
		\Big\|\sum_{a \in A_q} G(\chi, a) \calF^{-1}\big(\eta_s(\cdot - a/q) \hat{f} \big)\Big\|_{\ell^1} \\
		&\leq
		\Big\|\sum_{a \in A_q} G(\chi, a) \calF^{-1}\big(\eta_s(\cdot - a/q)\big) \Big\|_{\ell^1} \|f\|_{\ell^1},
	\end{align*}
	and
	\[
		\Big\|\sum_{a \in A_q} G(\chi, a) \calF^{-1}\big(\eta_s(\cdot - a/q)\big) \Big\|_{\ell^1}
		=
		\Big\|
		\sum_{a \in A_q} G(\chi, a) e^{2\pi i x a / q} \calF^{-1}\big(\eta_s\big)(x) \Big\|_{\ell^1(x)}.
	\]
	Now, by Theorem \ref{thm:7}, we can compute
	\begin{align*}
		\Big\|
        \sum_{a \in A_q} G(\chi, a) e^{2\pi i x a / q} \calF^{-1}\big(\eta_s\big)(x) \Big\|_{\ell^1(x)}
		&=
		\sum_{r \mid q/q_0} 
		\Big\|
        \sum_{a \in A_q} G(\chi, a) e^{2\pi i r a / q} \calF^{-1}\big(\eta_s\big)(q x + r) 
		\Big\|_{\ell^1(x)} \\
		&\leq
		q_0
		\sum_{r \mid q/q_0}
		\frac{\varphi(r)}{\varphi(q)}
		\big\|
		\calF^{-1}\big(\eta_s\big)(q x + r)
        \big\|_{\ell^1(x)} \\
		&\lesssim
		\frac{q_0}{q} \sum_{r \mid q/q_0} \frac{\varphi(r)}{\varphi(q)},
	\end{align*}
	where in the last inequality we have used Lemma \ref{lem:2} together with Lemma \ref{lem:1}. Since 
	(see e.g. \cite{sita})
	\[
		\sum_{r \mid q/q_0} \varphi(r) = \frac{q}{q_0},
	\]
	we conclude that
	\[
		\Big\|
        \sum_{a \in A_q} G(\chi, a) e^{2\pi i x a / q} \calF^{-1}\big(\eta_s\big)(x) \Big\|_{\ell^1(x)}
		\lesssim
		\frac{1}{\varphi(q)},
	\]
	proving the claim and the theorem follows.
\end{proof}

\begin{theorem}
	\label{thm:5}
	There is $C > 0$ such that for any subset $F \subset \ZZ$ of a finite cardinality and all $0 < \lambda < 1$,
	\[
		\Big|
		\Big\{x \in \ZZ : 
		\sup_{n \in \NN} \calM_{2^n} (\ind{F})(x) > \lambda
		\Big\}
		\Big|
		\leq
		C
		\lambda^{-1} \log^2\big(e/ \lambda\big) \abs{F}.
	\]
\end{theorem}
\begin{proof}
	We start by proving the following statement.
	\begin{claim}
		\label{clm:4}
		There are $C, c > 0$ such that for each $t > 0$, there are two sequences of operators
		$(A_n^t : n \in \NN)$ and $(B_n^t : n \in \NN)$ such that $\calM_{2^n} = A_n^t + B_n^t$, and
		for any finitely supported function $f: \ZZ \rightarrow \CC$,
		\begin{equation}
			\label{eq:23a}
			\sup_{\lambda > 0}{
			\lambda \cdot 
			\Big|\Big\{x \in \ZZ : 
			\sup_{n \in \NN} \big|A_n^t f(x) \big| > \lambda \Big\}\Big|}
			\leq C t \|f\|_{\ell^1},
		\end{equation}
		and
		\begin{equation}
			\label{eq:23b}
			\Big\|
			\sup_{n \in \NN}
			\big|
			B_n^t f
			\big|
			\Big\|_{\ell^2}
			\leq
			C \exp\big(- c \sqrt{t}\big) \|f\|_{\ell^2}.
		\end{equation}
	\end{claim}
	Without loss of generality, we may assume that $f$ is non-negative finitely supported function on $\ZZ$. 
	For $1 \leq n < t$, we set
	\[
		A_n^t f = \calM_{2^n} f,
		\qquad\text{and}\qquad
		B_n^t f \equiv 0.
	\]
	Since by the prime number theorem,
	\[
		\frac{2^n}{C} \leq \vartheta(2^n),
	\]
	we have
	\[
		\calM_{2^n} f(x) \leq C n M_{2^n} f (x).
	\]
	Hence, by the Hardy--Littlewood theorem,
	\begin{align*}
		\Big|\Big\{x \in \ZZ : 
		\sup_{1 \leq n < t} \calM_{2^n} f(x)  > \lambda
		\Big\}\Big|
		& \leq
		\Big|\Big\{x \in \ZZ : 
	    \sup_{1 \leq n < t} M_{2^n} f(x) > \frac{\lambda}{C t}
	    \Big\}\Big| \\
		&\lesssim
		\lambda^{-1} t \|f\|_{\ell^1}.
	\end{align*}
	For $t \leq n$, we set
	\[
		A_n^t f = \calF^{-1}\big(\Pi_n^t \hat{f} \big), \qquad\text{and}\qquad
		B_n^t f = \calM_{2^n} f - A_n^t f.
	\]
	In view of Corollary \ref{cor:3} and Theorem \ref{thm:3}, we obtain \eqref{eq:23b} and \eqref{eq:23a},
	respectively, and the claim follows.

	Now, the theorem is an easy consequence of Claim \ref{clm:4}. Indeed, given a subset $F \subset \ZZ$
	of a finite cardinality, for any $t > 0$, we can write
	\begin{align*}
		\Big|
        \Big\{
        \sup_{n \in \NN} \calM_{2^n} (\ind{F}) > \lambda
        \Big\}
        \Big|
		&\lesssim
		\Big|
        \Big\{
        \sup_{n \in \NN} \big| A_n^t (\ind{F})\big| > \tfrac{1}{2} \lambda
        \Big\}
        \Big|
		+
		\Big|
        \Big\{
        \sup_{n \in \NN} \big| B_n^t (\ind{F})\big| > \tfrac{1}{2} \lambda
        \Big\}
        \Big| \\
		&\lesssim
		\lambda^{-1} t \abs{F} + \lambda^{-2} \exp\big(-2 c \sqrt{t} \big) \abs{F}.
	\end{align*}
	Thus, taking 
	\[
		t = (2c)^{-2} \log^2 (e/\lambda),
	\]
	we get the desired conclusion.
\end{proof}

In view of \eqref{eq:39}, Theorem \ref{thm:5} entails the following corollary, which is precisely Theorem \ref{main_thm:2}.
\begin{corollary}
	\label{cor:2}
    There is $C > 0$ such that for any subset $F \subset \ZZ$ of a finite cardinality and all $0 < \lambda < 1$,
	\[
        \Big|
        \Big\{x \in \ZZ:
        \sup_{n \in \NN} \calA_{2^n} (\ind{F})(x) > \lambda
        \Big\}
        \Big|
        \leq
        C
        \lambda^{-1} \log^2\big(e/ \lambda\big) \abs{F}.
	\]
\end{corollary}

\section{Applications}
\label{sec:6}
In this section we show two applications of Theorem \ref{thm:5} and Corollary \ref{cor:2}. First, we prove that the
restricted weak Orlicz estimates together with strong $\ell^2$ bounds are sufficient to get $\ell^p$ maximal inequalities
for all $1 < p \leq 2$. Next, we conclude almost everywhere convergence of ergodic averages for functions in some Orlicz 
space close to $L^1$.
\subsection{$\ell^p$ theory}
\begin{theorem}
	\label{thm:8}
	For each $p \in (1, 2]$ there is $C > 0$ such that for any function $f \in \ell^p(\ZZ)$,
	\[
		\Big\|
		\sup_{N \in \NN} \big| \calM_N f \big|
		\Big\|_{\ell^p}
		\leq
		C (p-1)^{-4} \|f\|_{\ell^p}.
	\]
\end{theorem}
\begin{proof}
	With loss of generality, we may restrict the supremum to dyadic numbers. We claim the following holds true.
	\begin{claim}
		\label{clm:3}
		There is $C > 0$ such that for any subset $F \subset \ZZ$ of finite cardinality, and any $p_0 \in (1, \infty)$,
		\[
       		\sup_{\lambda > 0}
        	\lambda
			\cdot
       		\Big|
        	\Big\{x \in \ZZ : 
        	\sup_{n \in \NN} \calM_{2^n} (\ind{F})(x) > \lambda
        	\Big\}
        	\Big|^{\frac{1}{p_0}}
        	\leq
       		C
        	(p_0-1)^{-\frac{2}{p_0}} \abs{F}^{\frac{1}{p_0}}.
		\]
	\end{claim}
	Since $\calM_N$ are averaging operators, we may assume that $0 < \lambda < 1$. Observe that the function
	\[
		(0, 1) \ni \lambda \mapsto \lambda^{p_0-1} \log^2(e/\lambda)
	\]
	attains its maximum at
	\[
		\lambda = \exp\bigg(1-\frac{2}{p_0-1}\bigg).
	\]
	The maximal value equals $4 e^{p_0-3} (p_0-1)^{-2}$, thus
	\[
		\lambda^{-1} \log^2(e/\lambda) \leq 4 e^{p_0-3} (p_0-1)^{-2} \lambda^{-p_0}.
	\]
	Hence, by Theorem \ref{thm:5}, we get
	\begin{align*}
		\Big|
        \Big\{x \in \ZZ : 
        \sup_{n \in \NN} \calM_{2^n} (\ind{F})(x) > \lambda
        \Big\}
        \Big|
		&\leq
		C \lambda^{-1} \log^2(e/\lambda) \abs{F} \\
		&\leq 4 C e^{p_0-3} (p_0-1)^{-2} \lambda^{-p_0} \abs{F},
	\end{align*}
	which is what we claimed. 

	Next, we notice that by Theorem \ref{thm:1} and Theorem \ref{thm:2}, we have
	\begin{equation}
		\label{eq:31}
		\Big\|\sup_{n \in \NN} \big| \calM_{2^n} f \big| \Big\|_{\ell^2} \leq C \|f\|_{\ell^2}.
	\end{equation}
	Let us consider $p \in (1, 2)$. Set $p_0 = (1+p)/2$. Since $p_0 > 1$, the weak $\ell^{p_0}(\ZZ)$ is normable
	(see \cite{hunt}), thus at the cost of the additional factor of $(p-1)^{-1}$, we get
	\begin{equation}
		\label{eq:32}
		\sup_{\lambda > 0}
        \lambda
		\cdot
        \Big|
        \Big\{x \in \ZZ : 
        \sup_{n \in \NN} \big| \calM_{2^n} f (x) \big| > \lambda
        \Big\}
        \Big|^{\frac{1}{p}}
		\leq
		C (p-1)^{-1-\frac{2}{p}} \|f\|_{\ell^{p, 1}}
	\end{equation}
	for any $f \in \ell^{p,1}(\ZZ)$. Now, by the Marcinkiewicz interpolation theorem, \cite[Theorem 11.9]{adr}, based on 
	\eqref{eq:31} and \eqref{eq:32} we obtain 
	\[
		\Big\|\sup_{n \in \NN} \big| \calM_{2^n} f \big| \Big\|_{\ell^p} 
		\leq
		C
		\frac{p(2-p_0)}{(p-p_0)(2-p)} (p-1)^{-\frac{p+2}{p} \theta} \|f\|_{\ell^p}
	\]
	where $\theta \in (0, 1)$ satisfies
	\[
		\frac{1}{p} = \frac{\theta}{p_0} + \frac{1-\theta}{2}.
	\]
	Since 
	\[
		\frac{p(2-p_0)}{(p-p_0)(2-p)} (p-1)^{-\frac{p+2}{p} \theta} 
		= \frac{p(3-p)}{(p-1)(2-p)} (p-1)^{-\frac{p+2}{p} \cdot \frac{p+2-p^2}{p(3-p)}}
		\lesssim (p-1)^{-4},
	\]
	the theorem follows.
\end{proof}

\subsection{Pointwise convergence}
Let $(X, \calB, \mu)$ be a probability space with a measurable and measure preserving transformation $T: X \rightarrow X$.
We consider the following averages
\[
	\scrA_N f(x) = \frac{1}{\pi(N)} \sum_{p \in \PP_N} f\big(T^p x\big), \qquad x \in X.
\]
With a help of the Calder\'on transference principle from \cite{cald} applied to Corollary \ref{cor:2}, we deduce the
following proposition.
\begin{proposition}
	\label{prop:2}
	There is $C > 0$ such that for any subset $A \in \calB$, and all $0 < \lambda < 1$,
	\[
		\mu\Big\{x \in X : \sup_{N \in \NN} \scrA_N \big(\ind{A}\big)(x) > \lambda\Big\}
		\leq
		C \lambda^{-1} \log^2\big(e/\lambda\big) \mu(A).
	\]
\end{proposition}
\begin{proof}
	Fix $A \in \calB$ and $x \in X$. For $R > L > 0$, we define a finite subset of $F \subset \ZZ$ by setting
	\[
		F = \big\{0 \leq n \leq R : T^n x \in A \big\}. 
	\]
	Then for $0 \leq n \leq R - N$, $N \leq L$,
	\begin{align*}
		\scrA_N \big(\ind{A}\big)\big(T^n x\big) 
		&= \frac{1}{\pi(N)} \sum_{p \in \PP_N} \ind{A}\big(T^{n+p} x\big) \\
		&= \frac{1}{\pi(N)} \sum_{p \in \PP_N} \ind{F}(n + p)
		=
		\calA_N \big(\ind{F}\big)(n).
	\end{align*}
	Hence,
	\[
		\Big|\Big\{
		0 \leq n \leq R-L :
		\max_{1 \leq N \leq L}
		\scrA_N \big(\ind{A}\big)\big(T^n x\big) > \lambda
		\Big\}
		\Big|
		\leq
		\Big|\Big\{
		n \in \ZZ : \max_{1 \leq N \leq L} \calA_N \big(\ind{F}\big)(n) > \lambda
		\Big\}\Big|.
	\]
	By Corollary \ref{cor:2},
	\begin{align*}
		\Big|\Big\{
        n \in \ZZ : \max_{1 \leq N \leq L} \calA_N \big(\ind{F}\big)(n) > \lambda
        \Big\}\Big|
		&\leq
		C \lambda^{-1} \log^2(e/\lambda)
		\sum_{n \in \ZZ} \ind{F}(n) \\
		&=
		C \lambda^{-1} \log^2(e/\lambda) \sum_{n = 0}^R \ind{A}\big(T^n x\big).
	\end{align*}
	Since $T$ preserves the measure $\mu$, by integrating with respect to $x \in X$ we obtain
	\begin{align*}
		&
		(R-L+1) \cdot \mu\Big(x \in X : \max_{1 \leq N \leq L}
        \scrA_N \big(\ind{A}\big)(x) > \lambda \Big) \\
		&\qquad\qquad=
		\sum_{n = 0}^{R-L} \mu\Big(x \in X : \max_{1 \leq N \leq L}
        \scrA_N \big(\ind{A}\big)\big(T^n x\big) > \lambda \Big) \\
		&\qquad\qquad=
		\int_X \Big|\Big\{
		0 \leq n \leq R-L :
        \max_{1 \leq N \leq L}
        \scrA_N \big(\ind{A}\big)\big(T^n x\big) > \lambda
        \Big\}
        \Big| {\: \rm d}\mu(x) \\
		&\qquad\qquad\leq
		C \lambda^{-1} \log^2(e/\lambda) 
		\sum_{n = 0}^R \int_X \ind{A}\big(T^n x\big) {\: \rm d} \mu(x) \\
		&\qquad\qquad= C (R+1) \lambda^{-1} \log^2(e/\lambda) \mu(A).
	\end{align*}
	We now divide by $R$ and take $R$ approaching infinity to get
	\[
		\mu\Big(x \in X : \max_{1 \leq N \leq L}
        \scrA_N \big(\ind{A}\big)\big(T^n x\big) > \lambda \Big)
		\leq
		C \lambda^{-1} \log^2(e/\lambda) \mu(A).
	\]
	Finally, taking $L$ tending to infinity by the monotone convergence theorem we conclude the proof.
\end{proof}
We are now in the position to show $\mu$-almost everywhere convergence of the ergodic averages $(\scrA_N f : N)$ for 
a function $f$ from the Orlicz space $L(\log L)^2(\log \log L)(X, \mu)$. Let us recall that
$L(\log L)^2(\log \log L)(X, \mu)$ consists of functions such that
\[
	\int_X \abs{f(x)} \big(\log^+ \abs{f(x)} \big)^2 \big(\log^+\log^+ \abs{f(x)}\big) {\: \rm d} \mu(x) < \infty
\]
where $\log^+ t = \max\{0, \log t\}$. The space $L(\log L)^2(\log \log L)(X, \mu)$ is a Banach space with the norm
\[
	\big\|f\big\|_{L(\log L)^2(\log \log L)} =
	\int_0^1 f^*(t)
	\phi\big(t^{-1} \big) 
	{\: \rm d} t
\]
where $f^*$ is the decreasing rearrangement of $f$, that is
\[
	f^*(t) = \inf\Big\{s > 0 : \mu\big\{x \in X : \abs{f(x)} \geq s\big\} \leq t \Big\},
\]
and
\[
	\phi(t) = \log^2 (1 + t) \log \big(1 + \log t\big).
\]
\begin{theorem}
	\label{thm:4}
	There is $C > 0$ such that for each $f \in L(\log L)^2(\log \log L)(X, \mu)$, 
	\[
		\sup_{\lambda > 0}{\lambda \cdot
		\mu\Big\{x \in X : \sup_{N \in \NN} \big| \scrA_N f(x)\big| > \lambda\Big\}}
		\leq
		C
		\big\|f \big\|_{L(\log L)^2(\log \log L)}.
	\]
	In particular, for each $f \in L(\log L)^2(\log \log L)(X, \mu)$,
	\[
		\text{the limit} \quad
		\lim_{N \to \infty} \scrA_N f(x)
		\quad
		\text{exists}
	\]
	for $\mu$-almost all $x \in X$.
\end{theorem}
\begin{proof}
	We first prove the following claim.
	\begin{claim}
		\label{clm:2}
		There is $C > 0$ such that for each $A \in \calB$, and any $0 < \lambda < 1$,
		\begin{equation}
			\label{eq:17}
			\sup_{\lambda > 0}{
			\lambda \cdot \mu\Big\{x \in X : \sup_{N \in \NN} \scrA_N \big(\ind{A}\big)(x) > \lambda\Big\}}
			\leq
			C \mu(A) \log^2\bigg(\frac{e}{\mu(A)}\bigg).
		\end{equation}
	\end{claim}
	Indeed, by monotonicity, if $\lambda \geq \mu(A)$, then
	\begin{equation}
		\label{eq:18}
		\lambda^{-1} \mu(A) \log^2\bigg(\frac{e}{\lambda}\bigg) 
		\leq \lambda^{-1} \mu(A) \log^2\bigg(\frac{e}{\mu(A)} \bigg).
	\end{equation}
	Otherwise, $\lambda \leq \mu(A)$, which entails that
	\begin{equation}
		\label{eq:19}
		1 \leq \lambda^{-1} \mu(A) \leq \lambda^{-1} \mu(A) \log^2\bigg(\frac{e}{\mu(A)}\bigg).
	\end{equation}
	In view of Proposition \ref{prop:2},
	\[
		\mu\Big\{x \in X : \sup_{N \in \NN} \scrA_N \big(\ind{A}\big)(x) > \lambda\Big\}
		\leq
		\min\Big\{1, C \mu(A) \lambda^{-1} \log^2(e/\lambda)\Big\},
	\]
	which together with \eqref{eq:18} and \eqref{eq:19} easily lead to \eqref{eq:17}.

	Now, to show the theorem, let us fix $f \in L(\log L)^2 (\log \log L)(X, \mu)$. We set
	\[
		A_j = \Big\{x \in X : f^*(2^{-j+1}) < \abs{f(x)} \leq f^*(2^{-j}) \Big\},
	\]
	and
	\[
		a_j = f^*(2^{-j}).
	\]
	Since $\abs{f(x)} \leq a_j$ for $x \in A_j$, we have
	\[
		\abs{f} \leq \sum_{j \geq 1} a_j \ind{A_j}.
	\]
	Moreover, if $j > k$ then for $x \in A_j$ and $y \in A_k$, we have $\abs{f(x)} \geq \abs{f(y)}$. Since
	$\mu(A_j) = 2^{-j}$, we get
	\begin{equation}
		\label{eq:20}
		f^*(t) \geq \sum_{j \geq 1} a_j \ind{[2^{-j-1}, 2^{-j})}(t).
	\end{equation}
	Because the space $L^{1, \infty}(X, \mu)$ is log-convex (see \cite{kal, stwa}), by Claim \ref{clm:2}, we get
	\begin{align}
		\nonumber
		\sup_{\lambda > 0}{\lambda \cdot
		\mu\Big\{x \in X : \sup_{N \in \NN} \big| \scrA_N f(x)\big| > \lambda\Big\}}
		&\lesssim
		\sum_{j \geq 1} \log (j+1)
		\sup_{\lambda > 0}{\lambda \cdot
		\mu\Big\{x \in X : a_j \sup_{N \in \NN} \scrA_N \big(\ind{A_j}\big)(x) > \lambda\Big\}} \\
		\label{eq:21}
		&\lesssim
		\sum_{j \geq 1} \log (j+1) a_j \mu(A_j) \log^2\bigg(\frac{e}{\mu(A_j)}\bigg).
	\end{align}
	On the other hand, by \eqref{eq:20} we have  
	\begin{align*}
		\big\|f \big\|_{L(\log L)^2(\log \log L)} =
		\int_0^1 f^*(t) \phi\big(t^{-1}\big) {\: \rm d} t
		&\geq
		\sum_{j \geq 1}
		a_j
		\phi(2^j) 2^{-j-1} \\
		&\geq
		\frac{1}{8}
		\sum_{j \geq 1} a_j \mu(A_j) \log^2\bigg(\frac{e}{\mu(A_j)}\bigg) \log (j+1),
	\end{align*}
	which together with \eqref{eq:21} conclude the proof.
\end{proof}

\begin{bibliography}{discrete}
        \bibliographystyle{amsplain}
\end{bibliography}

\end{document}